\documentclass{amsart}
\usepackage{hyperref}

\usepackage{amsfonts,amsthm,amsmath}
\usepackage[shortlabels]{enumitem}
\setlist[enumerate,1]{label={\upshape(\roman*)}}

\usepackage{amssymb}

\theoremstyle{plain}
\newtheorem{thm}{Theorem}[section]
\newtheorem{lem}[thm]{Lemma}
\newtheorem{cor}[thm]{Corollary}

\theoremstyle{definition}
\newtheorem{df}[thm]{Definition}

\newcommand{\Frame}{\mathcal{F}(n)}
\newcommand{\Window}{\mathcal{W}(n)}

\title{Convex subgraphs and spanning trees of the square cycles}

\author[Munemasa]{Akihiro Munemasa}
\address{Graduate School of Information Sciences, Tohoku University, Sendai, 980-8579, Japan}
\email{munemasa@math.is.tohoku.ac.jp}

\author[Tanaka]{Yuuho Tanaka$^\ast$}
\address{Graduate School of Science and Engineering, Waseda University, Tokyo, 169-8555, Japan}
\email{tanaka\_yuuho\_dc@akane.waseda.jp}

\thanks {*Corresponding author}

\date{February 13, 2013}

\begin{document}


\begin{abstract}
We classify connected spanning convex subgraphs of the square cycles.
We then show that every spanning tree of $C_n^2$ is contained in a unique nontrivial connected spanning convex subgraph of $C_n^2$.
As a result, we obtain a purely combinatorial derivation of the formula for the number of spanning trees of the square cycles.
\end{abstract}
\maketitle

\section{Introduction}
It is well known that the number $t(G)$ of spanning trees of a connected graph $G$ can be computed using the matrix-tree theorem
(see e.g., \cite[Section~13.2]{GR}).
More precisely, $t(G)$ is the product of
nonzero eigenvalues of the Laplacian of $G$, divided by the number of vertices of $G$. 
For families of graphs whose Laplacian eigenvalues can be computed, this method is very useful in computing $t(G)$, except that
the results sometimes need to be simplified  since eigenvalues may not be rational integers.
Extensive work has been done to simplify the formula for $t(G)$ for circulant graphs
(see \cite{Med,Med2,Yong}).
For example, the derivation of the number 
$t(C_n^2)$ of spanning trees of the square cycle with $n$ vertices 
using the matrix-tree theorem was done first by Baron et al.~\cite{Baron}.
Kleitman and  Golden \cite{Kleitman} 
used a different approach to compute $t(C_n^2)$.
Namely, they used topological properties of a planar embedding of $C_n^2$ to derive a formula
for $t(C_n^2)$
when $n$ is even, and
mentioned that a similar method can be used to derive the same formula for odd $n$, without giving details.
If $n$ is even, $C_n^2$ is isomorphic to the rose window graph $R_{{n}/{2}}(1,1)$ \cite{Wilson}.
 The graph $C_n^2$ is also
denoted by $C_n(1,2)$ \cite{Med} and $C_n^{1,2}$ \cite{Yong}.



In this paper, 
we transform the topological argument given by
Kleitman and Golden \cite{Kleitman} to a purely combinatorial one, using the theory of graph homotopy \cite{Lewis}.
This allows us to give a uniform proof of the formula for $t(C_n^2)$ independent of the parity of $n$.
The key idea in our proof is the fact that every spanning tree of $C_n^2$ is contained in a unique nontrivial connected spanning
convex subgraph. Although this fact appeared implicitly in \cite{Kleitman} when $n$ is even, the classification of connected convex subgraphs of $C_n^2$ is new.


The organization of this paper is as follows.
In Section~\ref{sec:2}, we fix notation for the square cycles as circulant graphs,
and give some properties of the Fibonacci sequence.
We give a classification of connected spanning convex subgraphs of $C_n^2$ in Section~\ref{sec:3}.
In Section~\ref{sec:4}, we show that the set of the spanning trees of $C_n^2$
coincides with the disjoint union of the set of the spanning trees of strip graphs with tails $S_{n,k,j}$, defined in Section~\ref{sec:2}.
As a consequence, we deduce a combinatorial  proof of the formula for $t(C_n^2)$
which does not depend on the parity of $n$.

\section{Preliminaries}\label{sec:2}
\begin{df}
A graph that is connected and has no closed paths is called a \emph{tree}.
For a graph $G$, we say that $G'$ satisfying
\[
E(G')\subseteq E(G),V(G)=V(G')
\]
is a \emph{spanning subgraph} of $G$.
If a spanning subgraph $G'$ in a connected graph $G$ is a tree, then $G'$ is called a \emph{spanning tree} of the graph $G$.
\end{df}

\begin{df}
Let $n$ be an integer with $n\geq5$.
The \emph{square cycle} $C_n^2$ is defined by $V(C^2_n)= \mathbb{Z}_n=\mathbb{Z}/n\mathbb{Z}$, $E(C^2_n)=\{\{v_i,v_j\}\mid v_i,v_j\in V(C_n^2),\;i,j\in\mathbb{Z},\;i-j=1,2\}$, where $v_i=i+n\mathbb{Z}\in\mathbb{Z}_n$.
\end{df}

Let $n$ be an integer with $n\geq5$.
Then, $E(C_n^2)=\{e_i \mid i\in\mathbb{Z}\}\cup\{f_i\mid i\in\mathbb{Z}\}$, where we define \emph{frame} $e_i$ and \emph{window} $f_i$ as follows.
\[
e_i=\{v_i,v_{i+1}\},\;f_i=\{v_i,v_{i+2}\} \quad (i\in\mathbb{Z}).
\]
We denote by $\Window$ and $\Frame$ the set of frames and windows, respectively as follows.
\begin{align*}
\Window&=\{f_i\mid 0\leq i\leq n-1\},\\
\Frame&=\{e_i \mid 1\leq i\leq n\}.
\end{align*}

By a \emph{triangle} of $C_n^2$ we mean a set
\[
T_i=\{e_i,e_{i+1},f_i\} \quad \text{($i\in\mathbb{Z}$).}
\]
Then,
\[
E(C_n^2)=\bigcup_{i=0}^{n-1}T_i.
\]

\begin{df}
Given $i$ ($i\in\mathbb{Z}$), if a subgraph $G$ of $C_n^2$ satisfies $|T_i\cap E(G)|\leq1$ or $T_i\subseteq E(G)$, then $G$ is said to be \emph{convex} with respect to the triangle $T_i$.
A subgraph $G$ of $C_n^2$ is said to be \emph{convex} if $G$ is convex with respected to $T_i$ for all $i$ ($i\in\mathbb{Z}$).
\end{df}

\begin{df}
The graph $S_k$ defined by $V(S_k)=\{1,2,\dots,k\}$,
$E(S_k)=\{\{i,j\} \mid i,j\in V(S_k),1\leq |i-j|\leq 2\}$ 
is called a \emph{strip graph}.
\end{df}

The sequence of numbers $F_n$ defined by the recurrence relation $F_0=0,F_1=1,F_{n+2}=F_{n+1}+F_n$
($n=0,1,2,\dots$) is called the \emph{Fibonacci sequence}.
The following two lemmas are due to 
Kleitman and Golden \cite{Kleitman}.

\begin{lem}
\label{lem:kleitman1}
For $n\geq2$, $t(S_n)=F_{2n-2}$.
\end{lem}

\begin{lem}
\label{fib3}
For $n\geq2$,
\[
F_n^2=\begin{cases}
\sum_{k=0}^{(n-2)/2}F_{4k+2} &\text{if $n$ is even,}\\
1+\sum_{k=1}^{(n-1)/2}F_{4k} &\text{if $n$ is odd.}
\end{cases}
\]
\end{lem}

The following substructures appeared implicitly 
in \cite{Kleitman}. In fact, an escape route
is the set of edges crossed by a path from the
interior to the outside region,  in the planar  drawing
of $C_n^2$ (see \cite[Fig.~4]{Kleitman}).
The removal of an escape route gives a strip graph
with tails (see \cite[Fig.~5]{Kleitman}).

\begin{df}
Let $n\geq5$. For integers $j$ and $k$ with
$0\leq k\leq \lceil\frac{n-2}{2}\rceil$, we define 
the graph $S_{n,k,j}$ as follows:
\begin{align*}
V(S_{n,k,j})&=V(C_n^2),\\
E(S_{n,k,j})&=E(C_n^2)\setminus ES(n,k,j), 
\quad (j,k\in\mathbb{Z},\;0\leq k\leq \lceil\frac{n-2}{2}\rceil),
\end{align*}
where
\[
ES(n,k,j)=\{f_j,f_{j+2k+1}\}\cup\{e_{j+1},\dots,e_{j+2k+1}\} \quad (j,k\in\mathbb{Z},\;0\leq k\leq \lceil\frac{n-2}{2}\rceil).
\]
The graph $S_{n,k,j}$ is called a \emph{strip graph with tails}, and $ES(n,k,j)$ is called the \emph{escape route}.
\end{df}

The graphs $S_{n,k,j}$ are connected spanning convex subgraphs of $C_n^2$.
Clearly, $C_n^2$ and $(\mathbb{Z}_n,\Window)$ for $n$ odd are also connected spanning subgraphs of $C_n^2$, and we call these subgraphs \emph{trivial} connected spanning subgraphs.

For a graph $G$, let $T_G$ be the set of all spanning trees of $G$.
Then $t(G)=|T_G|$.
Since $S_{n,k,j}$ can be obtained from the strip graph $S_{n-2k}$ by attaching two tails of length $k$, the following lemma holds.

\begin{lem} \label{lem:strip_tail}
For $j,k\in\mathbb{Z},\;0\leq k\leq \lceil\frac{n-2}{2}\rceil$, $t(S_{n,k,j})=t(S_{n-2k})$.
\end{lem}

\section{Spanning convex subgraphs}\label{sec:3}

In this section, we prove our first main result which gives a classification of connected spanning convex subgraphs of $C_n^2$.

\begin{lem} \label{lem:pn}
Let $G$ be a connected spanning convex subgraph of $C_n^2$.
If $k$ and $p$ are integers with $0\leq p<n$ and
\begin{equation} \label{lemeq:k}
\{e_{k-1},f_k,f_{k+2},\dots,f_{k+2p-2},e_{k+2p}\}\subseteq E(G),
\end{equation}
then $\{e_k,e_{k+1},\dots,e_{k+2p-1}\}\subseteq E(G)$.
\end{lem}

\begin{proof}
We prove the assertion by induction on $p$.
If $p=0$, then it is trivial.
Therefore, we may assume that $p\geq1$.

Suppose that there exists an integer $i$ with $0\leq i\leq 2p-1$ such that $e_{k+i}\in E(G)$.
If $i$ is even, then since $G$ is convex with respected to $T_{k+i}$, $e_{k+i+1}\in E(G)$.
Therefore, we can apply the induction to $\{e_{k-1},f_k,f_{k+2},\dots,f_{k+i-2},e_{k+i}\}$ and $\{e_{k+i+1},f_{k+i+2},f_{k+i+4},\dots,f_{k+2p-2},e_{k+2p}\}$.
Similarly, if $i$ is odd, then 
we can apply the induction. 

It remains to derive a contradiction by assuming
\begin{equation} \label{eq:k3}
e_{k},e_{k+1},\dots,e_{k+2p-1}\notin E(G).
\end{equation}
Since $G$ is convex with respect to $T_{k-1}$,
\begin{equation} \label{eq:k1}
f_{k-1}\notin E(G).
\end{equation}
Similarly, since $G$ is convex with respect to
$T_{k+2p-1}$
\begin{equation} \label{eq:k5}
f_{k+2p-1}\notin E(G).
\end{equation}
From 
\eqref{eq:k3}, 
\eqref{eq:k1}, 
and \eqref{eq:k5}, 
we see that the set
$\{v_{k+1},v_{k+3},\dots,v_{k+2p-1}\}$ is separated
from its complement in the connected spanning subgraph $G$. 
This is a  contradiction.
\end{proof}

\begin{lem} \label{lem:2-3}
Let $G$ be a nontrivial connected spanning convex subgraph of $C_n^2$.
If $E(G)$ contains no frame, then $n$ is odd, and $G=S_{n,\frac{n-1}{2},j}$ for some integer $j$ with $0\leq j\leq n-1$.
\end{lem}

\begin{proof}
By the assumption, $E(G)$ consists only of windows.
Since $G$ is connected, $n$ is odd.
Since $G$ is nontrivial, 
\(
|E(G)|\leq n-1.
\)
Since $G$ is connected, $|E(G)|\geq n-1$.
Therefore, $|E(G)|= n-1$.
Then there exists $j$ such that $E(G)=\Window\setminus\{f_j\}=E(S_{n,\frac{n-1}{2},j})$.
This proves $G=S_{n,\frac{n-1}{2},j}$.
\end{proof}

\begin{lem} \label{lem:2-2}
Let $G$ be a nontrivial connected spanning convex subgraph of $C_n^2$.
If $E(G)$ contains a frame, then $G=S_{n,k,j}$ for some integers $j,k$ with $0\leq j\leq n-1$, $0\leq k\leq\lfloor\frac{n-2}{2}\rfloor$.
\end{lem}

\begin{proof}
If $\Frame\subset E(G)$, then it is easy to see that $G=C_n^2$, contradicting the assumption that $G$ is nontrivial.
Since $\Frame\cap E(G)\neq\emptyset$, 
there exists $i,l$ with $0\leq i\leq n-1$, $1\leq l\leq n-1$ satisfying $\{e_i,e_{i+1},\dots,e_{i+l-1}\}\subseteq E(G)$ and $e_{i-1},e_{i+l}\notin E(G)$.
Without loss of generality, we may assume that $i=0$.
In this case, we have
\begin{align}
\{e_0,e_1,\dots,e_{l-1}\}&\subseteq E(G),  \label{eq:il1}\\
e_{-1}&\notin E(G),  \label{eq:il1-1}\\
e_l&\notin E(G).  \label{eq:il1-12}
\end{align}
Since $G$ is convex with respected to $T_j$ ($0\leq j\leq l-2$),
\eqref{eq:il1} implies
\begin{equation} \label{eq:il11}
f_0,f_1,\dots,f_{l-2}\in E(G).
\end{equation}
Since $G$ is convex with respected to $T_{-1}$, \eqref{eq:il1} and \eqref{eq:il1-1} imply
\begin{equation} \label{eq:il12}
f_{-1}\notin E(G).
\end{equation}
Since $G$ is convex with respect to $T_{l-1}$, \eqref{eq:il1} and \eqref{eq:il1-12} imply
\begin{equation} \label{eq:il31}
f_{l-1}\notin E(G).
\end{equation}

Let $s$ and $t$ be the largest non-negative integers such that
\begin{equation} \label{eq:il2}
f_{-2},f_{-4},\dots,f_{-2s}\in E(G),
\end{equation}
and
\begin{equation} \label{eq:il3}
f_l,f_{l+2},\dots,f_{l+2t-2}\in E(G),
\end{equation}
respectively.
Then, $f_{-2s-2}\notin E(G)$ and $f_{l+2t}\notin E(G)$.

We show that
\begin{equation} \label{eq:tn2}
e_l,e_{l+1},\dots,e_{l+2t}\notin E(G)\text{ and }t<\frac{n-l}{2},
\end{equation}
\begin{equation} \label{eq:sn2}
e_{-1},e_{-2},\dots,e_{-2s-1}\notin E(G)\text{ and }s<\frac{n-l}{2}.
\end{equation}
Assume that there exists an integer $m$ with $0\leq m\leq 2t$ and that $e_{l+m}\in E(G)$.
We may choose minimal such $m$.
By \eqref{eq:il1-12}, we have $m>0$.
If $m$ is odd, then by \eqref{eq:il3} and by the convexity of $G$, 
$T_{l+m-1}\subseteq E(G)$.
Therefore, $e_{l+m-1}\in E(G)$. This contradicts the minimality of $m$. 
If $m$ is even, then by \eqref{eq:il1} and \eqref{eq:il3}, we have 
$\{e_{l-1},f_l,f_{l+2},\dots,f_{l+m-1},e_{l+m}\}\subseteq E(G)$.
Then, by Lemma~\ref{lem:pn}, 
we have $e_{l+m-1}\in E(G)$, again contradicting the minimality of $m$.
Therefore, \eqref{eq:tn2} holds.
Similarly, we can prove \eqref{eq:sn2}.

Let $K=\{v_{-2s},v_{-2s+2},\dots,v_0,v_1,\dots,v_l,v_{l+2},\dots,v_{l+2t}\}$.
If $K\neq\mathbb{Z}_n$,
then by \eqref{eq:il12}, \eqref{eq:il31}, \eqref{eq:tn2} and \eqref{eq:sn2}, $G$ is disconnected.
This is a  contradiction. 
Therefore, $K=\mathbb{Z}_n$, and in particular,  $s+l+1+t=|K|=n$.
From \eqref{eq:tn2} and \eqref{eq:sn2}, $s=t=\frac{n-l-1}{2}$.
Then, from \eqref{eq:il2}, \eqref{eq:il3} and \eqref{eq:tn2}, we have
\begin{align*}
\{f_{l+1},f_{l+3},\dots,f_{n-2}\}&\subseteq E(G),
\\
\{f_l,f_{l+2},\dots,f_{n-3}\}&\subseteq E(G),
\\
e_l,e_{l+1},\dots,e_{n-1}&\notin E(G),
\end{align*}
respectively.
Together with
\eqref{eq:il1}, \eqref{eq:il11},
\eqref{eq:il12} and
\eqref{eq:il31},
these imply
$E(G)=E(S_{n,\frac{n-l-1}{2},l-1})$.
This proves $G=S_{n,\frac{n-l-1}{2},l-1}$.
\end{proof}

\begin{thm} \label{th:4-4}
Let $G$ be a nontrivial connected spanning convex subgraph of $C_n^2$.
Then
there exists integers $j,k$ with $0\leq j\leq n-1$,
$0\leq k\leq\lceil\frac{n-2}{2}\rceil$ such that $G=S_{n,k,j}$.
\end{thm}

\begin{proof}
This is immediate from Lemmas \ref{lem:2-3} and \ref{lem:2-2}.
\end{proof}

\section{Enumerating spanning trees of the square cycles}\label{sec:4}

In this section, we prove our second main result which states that every spanning tree of $C_n^2$ is contained in a unique connected spanning convex subgraph. As a consequence, we obtain an alternative proof of the formula for the number of spanning trees of $C_n^2$.
Our method is a combinatorial formulation of the topological proof given in \cite{Kleitman}. 
The tool we use is the theory of graph homotopy. We refer the reader to \cite{Lewis} for the precise definition of the homotopy group. Roughly speaking, the homotopy group $\pi(G,v_0)$ of the graph $G$ with respect to a vertex $v_0$ is the group formed by equivalence classes of circuits through $v_0$. It contains the subgroup $\pi(G,v_0,3)$ which is ``generated'' by triangles. It is clear that $\pi(G,v_0)=\pi(G,v_0,3)$ if $G$ is a tree, strip graph, or strip graph with tails, while
$\pi(G,v_0)\neq\pi(G,v_0,3)$ if $G$ is a cycle of length at least $4$ or $G=C_n^2$ with $n\geq7$.

\begin{thm} \label{th:gsnkj}
Let $n$ be an integer with $n\geq5$.
For every is a spanning tree $G$ of $C_n^2$,
there exists a unique nontrivial connected spanning convex subgraph $H$ of $C_n^2$ such that $E(G)\subseteq E(H)$.
More precisely,
\begin{equation}\label{eq:dunion}
T_{C_n^2}=
\bigcup_{k=0}^{\lceil\frac{n-2}{2}\rceil}
\bigcup_{j=0}^{n-1}
T_{S_{n,k,j}} \quad \text{(disjoint).}
\end{equation}
\end{thm}

\begin{proof}
Since the assertion can be verified directly for $n=5$ and $6$, we assume $n\geq7$.
According to Lewis \cite{Lewis}, for a graph $G$ we can define its homotopy group $\pi(G,v_0)$ and the normal subgroup $\pi(G,v_0,3)$ of $\pi(G,v_0)$ generated by the triangles.
Clearly $\pi(G,v_0)$ is the trivial group for the spanning tree $G$ of $C_n^2$, so in particular $\pi(G,v_0)=\pi(G,v_0,3)$ holds.
For a spanning tree $G$ of $C_n^2$ which is not convex with respect to some triangle $T_i$, $\pi(G',v_0)=\pi(G',v_0,3)$ also holds for the graph $G'$ obtained from $G$ by adding the unique missing edge of $T_i$.
This process can be iterated until we reach a convex subgraph containing $G$.
The resulting graph $H$ is a connected spanning convex subgraph $H$ of $C_n^2$, and hence it is one of the graphs classified in Theorem \ref{th:4-4}, or one of the trivial 
connected spanning convex subgraph.
Since
\(\pi(H,v_0)=\pi(H,v_0,3)\)
holds only for nontrivial connected spanning convex subgraph $H$, there exist $j,k$ with $0\leq j\leq n-1$, $0\leq k\leq\lceil\frac{n-2}{2}\rceil$ such that $E(G)\subseteq E(S_{n,k,j})$.

It remains to show that the union
in \eqref{eq:dunion} is disjoint.
Suppose $E(G)\subseteq E(S_{n,k',j'})$ for some $j',k'$ with $0\leq k'\leq \lceil\frac{n-2}{2}\rceil,\;0\leq j'\leq n-1$.
Then the subgraph with edge set $E(S_{n,k,j})\cap E(S_{n,k',j'})$ is a nontrivial connected spanning convex subgraph of $C_n^2$, and hence coincides with $S_{n,k',j'}$ for some $j'',k''$ with $0\leq k''\leq \lceil\frac{n-2}{2}\rceil,\;0\leq j''\leq n-1$.
This implies $E(S_{n,k'',j''})\subseteq E(S_{n,k,j})$ which is possible only when $(j,k)=(j'',k'')$.
Then we have $(j,k)=(j',k')$.
Therefore,  the union
in \eqref{eq:dunion} is disjoint.
\end{proof}

\begin{cor}[Kleitman and Golden \cite{Kleitman}] \label{kleitman}
\[
t(C_n^2)=nF_n^2.
\]
\end{cor}

\begin{proof}
\begin{align*}
t(C_n^2)
&=\sum_{k=0}^{\lceil\frac{n-2}{2}\rceil}\sum_{j=0}^{n-1}t(S_{n,k,j}) &&\text{(by Theorem \ref{th:gsnkj})}
\displaybreak[0]\\
&=n\sum_{k=0}^{\lceil\frac{n-2}{2}\rceil}t(S_{n-2k}) &&\text{(by Lemma \ref{lem:strip_tail})}
\displaybreak[0]\\
&=\begin{cases}
n\sum_{k=0}^{(n-2)/2}t(S_{2k+2}) &\text{if $n$ is even,}\\
n+n\sum_{k=1}^{(n-1)/2}t(S_{2k+1}) &\text{if $n$ is odd}
\end{cases}
\displaybreak[0]\\
&=\begin{cases}
n\sum_{k=0}^{(n-2)/2}F_{4k+2} &\text{if $n$ is even,}\\
n(1+\sum_{k=1}^{(n-1)/2}F_{4k}) &\text{if $n$ is odd}
\end{cases} &&\text{(by Lemma \ref{lem:kleitman1})}
\displaybreak[0]\\
&=nF_n^2 &&\text{(by Lemma \ref{fib3}).}
\end{align*}
\end{proof}



\begin{thebibliography}{99}
\bibitem{Med} A. D. Mednykh, I. A. Mednykh. On rationality of generating function for the number of spanning trees in circulant graphs, {\sl Algebra Colloq. }{\bf27}, No.1, 87--94 (2020).
\bibitem{Med2} A. D. Mednykh, I. A. Mednykh. The number of spanning trees in circulant graphs, its arithmetic properties and asymptotic, {\sl Discrete Math. }{\bf342}, No.6, 1772--1781 (2019).
\bibitem{Baron} G. Baron, H. Prodinger, R. F. Tichy, F. T. Boesch, J. F. Wang, The number of spanning trees in the square of a cycle,
{\sl Fibonacci Quart. 23, no. 3}, 258--264 (1985).
\bibitem{GR} C. Godsil and G. Royle,
Algebraic Graph Theory, Springer, 2001.
\bibitem{Kleitman}  D. J. Kleitman, B. Golden. Counting trees in a certain class of graphs, {\sl The American Mathematical Monthly} {\bf82}, No.1, 40--44 (1975).
\bibitem{Lewis} H. A. Lewis. Homotopy in $Q$-polynomial distance-regular graphs, {\sl Discrete Math. }{\bf223}, 189--206 (2000).
\bibitem{Wilson} S. Wilson. Rose window graphs, {\sl Ars Mathematica Contemporanea 1}, 7--19 (2008).
\bibitem{Yong} X. Yong, T. Acenjian. The numbers of spanning trees of the cubic cycle $C_N^3$ and the quadruple cycle $C_N^4$, {\sl Discrete Math. }{\bf169}, 293--298 (1997).
\end{thebibliography}
\end{document}